\documentclass[10pt]{amsart}
\usepackage{amssymb}
\usepackage{dsfont}
\usepackage{amscd}
\usepackage[mathscr]{euscript} 
\usepackage[all]{xy}
\usepackage{url}
\usepackage{stmaryrd}
\usepackage{comment}
\usepackage[retainorgcmds]{IEEEtrantools}
\usepackage{hyperref}
\title{Model theory of fields with finite group scheme actions}
\author[D. M. HOFFMANN]{Daniel Max Hoffmann$^{\dagger}$}
\thanks{$^{\dagger}$SDG. The first author is supported by
the Polish National Agency for Academic Exchange.}
\address{$^{\dagger}$ Instytut Matematyki\\
Uniwersytet Warszawski\\
Warszawa\\
Poland
\newline \indent {\em and}\newline
\indent \hspace{1mm} Department of Mathematics \\ University of Notre Dame \\ Notre Dame \\ IN \\ USA}
\email{daniel.max.hoffmann@gmail.com}
\urladdr{https://sites.google.com/site/danielmaxhoffmann/home}
\author[P. KOWALSKI]{Piotr Kowalski$^{\spadesuit}$}
\thanks{$^{\spadesuit}$ Supported by the Narodowe Centrum Nauki grant no. 2018/31/B/ST1/00357 and by the T\"{u}bitak 1001 grant no. 119F397.}
\address{$^{\spadesuit}$Instytut Matematyczny\\
Uniwersytet Wroc{\l}awski\\
Wroc{\l}aw\\
Poland}
\email{pkowa@math.uni.wroc.pl} \urladdr{http://www.math.uni.wroc.pl/\textasciitilde pkowa/ }
\thanks{2010 \textit{Mathematics Subject Classification} Primary 03C60, 14L15 Secondary 11S20.}
\thanks{\textit{Key words and phrases}. Finite group scheme action, Model companion, Simple theory.}

\DeclareMathOperator{\acl}{acl}  
  \DeclareMathOperator{\id}{id}
 \DeclareMathOperator{\fr}{Fr} 
  
  \DeclareMathOperator{\gal}{Gal}
\DeclareMathOperator{\ch}{char}  
 
 \DeclareMathOperator{\alg}{alg}

\DeclareMathOperator{\spec}{Spec}\DeclareMathOperator{\rat}{rat}
\DeclareMathOperator{\sep}{sep}

\DeclareMathOperator{\dcf}{DCF}
\DeclareMathOperator{\mc}{mc}
\newtheorem{theorem}{Theorem}[section]
\newtheorem{prop}[theorem]{Proposition}
\newtheorem{lemma}[theorem]{Lemma}
\newtheorem{cor}[theorem]{Corollary}

\theoremstyle{definition}
\newtheorem{definition}[theorem]{Definition}
\newtheorem{example}[theorem]{Example}
\newtheorem{remark}[theorem]{Remark}
\newtheorem{question}[theorem]{Question}

\begin{document}
\newcommand{\fg}{\mathfrak{g}}
\newcommand{\lili}{\underleftarrow{\lim }}
\newcommand{\coco}{\underrightarrow{\lim }}
\newcommand{\twoc}[3]{ {#1} \choose {{#2}|{#3}}}
\newcommand{\thrc}[4]{ {#1} \choose {{#2}|{#3}|{#4}}}
\newcommand{\Zz}{{\mathds{Z}}}
\newcommand{\Ff}{{\mathds{F}}}
\newcommand{\Cc}{{\mathds{C}}}
\newcommand{\Rr}{{\mathds{R}}}
\newcommand{\Nn}{{\mathds{N}}}
\newcommand{\Qq}{{\mathds{Q}}}
\newcommand{\Kk}{{\mathds{K}}}
\newcommand{\Pp}{{\mathds{P}}}
\newcommand{\ddd}{\mathrm{d}}
\newcommand{\Aa}{\mathds{A}}
\newcommand{\dlog}{\mathrm{ld}}
\newcommand{\ga}{\mathbb{G}_{\rm{a}}}
\newcommand{\gm}{\mathbb{G}_{\rm{m}}}
\newcommand{\gaf}{\widehat{\mathbb{G}}_{\rm{a}}}
\newcommand{\gmf}{\widehat{\mathbb{G}}_{\rm{m}}}
\newcommand{\ka}{{\bf k}}
\newcommand{\ot}{\otimes}
\newcommand{\si}{\mbox{$\sigma$}}
\newcommand{\ks}{\mbox{$({\bf k},\sigma)$}}
\newcommand{\kg}{\mbox{${\bf k}[G]$}}
\newcommand{\ksg}{\mbox{$({\bf k}[G],\sigma)$}}
\newcommand{\ksgs}{\mbox{${\bf k}[G,\sigma_G]$}}
\newcommand{\cks}{\mbox{$\mathrm{Mod}_{({A},\sigma_A)}$}}
\newcommand{\ckg}{\mbox{$\mathrm{Mod}_{{\bf k}[G]}$}}
\newcommand{\cksg}{\mbox{$\mathrm{Mod}_{({A}[G],\sigma_A)}$}}
\newcommand{\cksgs}{\mbox{$\mathrm{Mod}_{({A}[G],\sigma_G)}$}}
\newcommand{\crats}{\mbox{$\mathrm{Mod}^{\rat}_{(\mathbf{G},\sigma_{\mathbf{G}})}$}}
\newcommand{\crat}{\mbox{$\mathrm{Mod}^{\rat}_{\mathbf{G}}$}}
\newcommand{\cratinv}{\mbox{$\mathrm{Mod}^{\rat}_{\mathbb{G}}$}}
\newcommand{\ra}{\longrightarrow}
\newcommand{\bdcf}{B-\dcf}
\makeatletter
\providecommand*{\cupdot}{%
  \mathbin{%
    \mathpalette\@cupdot{}%
  }%
}
\newcommand*{\@cupdot}[2]{%
  \ooalign{%
    $\m@th#1\cup$\cr
    \sbox0{$#1\cup$}%
    \dimen@=\ht0 %
    \sbox0{$\m@th#1\cdot$}%
    \advance\dimen@ by -\ht0 %
    \dimen@=.5\dimen@
    \hidewidth\raise\dimen@\box0\hidewidth
  }%
}

\providecommand*{\bigcupdot}{%
  \mathop{%
    \vphantom{\bigcup}%
    \mathpalette\@bigcupdot{}%
  }%
}
\newcommand*{\@bigcupdot}[2]{%
  \ooalign{%
    $\m@th#1\bigcup$\cr
    \sbox0{$#1\bigcup$}%
    \dimen@=\ht0 %
    \advance\dimen@ by -\dp0 %
    \sbox0{\scalebox{2}{$\m@th#1\cdot$}}%
    \advance\dimen@ by -\ht0 %
    \dimen@=.5\dimen@
    \hidewidth\raise\dimen@\box0\hidewidth
  }%
}
\makeatother

\def\Ind#1#2{#1\setbox0=\hbox{$#1x$}\kern\wd0\hbox to 0pt{\hss$#1\mid$\hss}
\lower.9\ht0\hbox to 0pt{\hss$#1\smile$\hss}\kern\wd0}

\def\ind{\mathop{\mathpalette\Ind{}}}

\def\notind#1#2{#1\setbox0=\hbox{$#1x$}\kern\wd0
\hbox to 0pt{\mathchardef\nn=12854\hss$#1\nn$\kern1.4\wd0\hss}
\hbox to 0pt{\hss$#1\mid$\hss}\lower.9\ht0 \hbox to 0pt{\hss$#1\smile$\hss}\kern\wd0}

\def\nind{\mathop{\mathpalette\notind{}}}

\maketitle
\begin{abstract}
We study model theory of fields with actions of a fixed finite group scheme. We prove the existence and simplicity of a model companion of the theory of such actions, which generalizes our previous results about truncated iterative Hasse-Schmidt derivations (\cite{HK}) and about Galois actions (\cite{HK3}). As an application of our methods, we obtain a new model complete theory of actions of a finite group on fields of finite imperfection degree.
\end{abstract}

\section{Introduction}
In this paper, we deal with model theory of actions of finite group schemes on fields. Our results here provide a generalization of both the results from \cite{HK} (model theory of truncated iterative Hasse-Schmidt derivations on fields) and the results from \cite{HK3} (model theory of actions of finite groups on fields). We explain below the nature of this generalization.

Let us fix a field $\ka$. By a finite group scheme over $\ka$, we mean the dual object to a finite-dimensional Hopf algebra over $\ka$ (see Section \ref{secprol}). Let us fix such a finite group scheme $\fg$. It is a classical result (\cite[Section 6.7]{Water}) that $\fg$ fits into the following short exact sequence:
$$1\ra \fg^0\ra \fg\ra \fg_{0}\ra 1,$$
where the finite group scheme $\fg^0$ is infinitesimal (that is: its underlying scheme is connected) and the finite group scheme $\fg_0$ is \'{e}tale (that is: it becomes a constant finite group scheme after a finite Galois base extension). As explained in \cite{HK}, the actions of infinitesimal finite group schemes on fields correspond to truncated iterative Hasse-Schmidt derivations. It is also clear that the actions of constant finite group schemes on fields correspond to actions of finite groups, and the model theory of such actions was analyzed in \cite{HK3}. It should be clear now that the model theory of actions of finite group schemes, which is the topic of this paper, encompasses the situations from \cite{HK} and \cite{HK3}.

Before describing the main results of this paper, let us compare the structures we consider here with a certain kind of \emph{operators} on fields. Moosa and Scanlon developed in \cite{MS1} a general theory of rings with iterative operators. They analyzed the model theory of fields with such operators in \cite{MS2} under two additional assumptions: the base field is of characteristic $0$ and the operators are \emph{free} (every free operator is inter-definable in a natural way with a certain iterative operator as explained in \cite{MS2}). The results from \cite{MS2} were extended to an arbitrary characteristic case (still in the free context) in \cite{BHKK}. The actions of finite group schemes, which are the topic of this paper, are the same as iterative operators in the sense of Moosa-Scanlon for \emph{constant} iterative systems (as we explain in Section \ref{secprol}), and they should be thought of as ``very non-free'' iterative operators.

We list here the main results of this paper.
\begin{enumerate}
\item The theory of actions of $\mathfrak{g}$ on fields has a model companion, which we denote by $\fg$-DCF (Theorem \ref{gdcfexists}).

\item The theory $\mathfrak{g}$-DCF is simple (Theorem \ref{simple}).

\item The theory of actions of a finite group $G$ on fields of characteristic $p>0$ and of finite imperfection degree $e$
has a model companion, which is bi-interpretable (after adding finitely many constants) with the theory $(\ga[1]\times G_{\Ff_p})$-DCF, where $\ga[1]$ is the kernel of the Frobenius endomorphism on the additive group scheme over the prime field $\Ff_p$ (Theorem \ref{appthm}).
\end{enumerate}
Coming back to the situation from \cite{MS2}, we should mention that one can not expect a generalization of the above results to the case of general iterative operators considered in  \cite{MS2}, since the theory of such operators on fields may be not companionable as it is shown in \cite{MS2} and \cite{BHKK} (we comment more on this issue in Remark \ref{ssrem}(5)).

This paper is organized as follows. In Section \ref{secprol}, we give the main definitions, provide some examples, and we also discuss the notion of a prolongation with respect to a fixed action of $\fg$. In Section \ref{secexist}, we show the existence of a model companion of the theory of $\fg$-actions on fields using the notion of a prolongation from Section \ref{secprol}. In Section \ref{secprop}, we analyze the model-theoretic properties of the theory obtained in Section  \ref{secexist}. In particular, we show by the Galois-theoretic methods that this theory is simple. As an application of the results from Sections \ref{secexist} and \ref{secprop}, we give in Section \ref{secapp} a new example of a theory fitting to the set-up considered by the first author in \cite{Hoff3}, that is we prove the existence of a model companion of the theory of actions of a finite group on fields of a fixed finite imperfection degree.

\section{Finite group scheme actions: first-order set-up and prolongations}\label{secprol}
In this section, we describe our set-up of finite group scheme actions both in the algebraic as well as in the model-theoretic context. Then, we define the notion of a prolongation and collect some technical tools, which will be needed in the sequel.

We fix for the rest of this paper a base field $\ka$.

\subsection{Finite group scheme actions: definitions and examples}\label{secdefex}
For the necessary background about affine group schemes, their actions, and Hopf algebras, we refer the reader to \cite{Water}. All the group schemes considered in this paper will be affine, so we will skip the adjective ``affine'' from now on. Let us fix a finite group scheme $\fg$ over $\ka$. Then $\fg=\spec(H)$, where $H$ is a finite dimensional Hopf algebra over $\ka$. We denote the comultiplication in $H$ by $\mu:H\to H\otimes_{\ka} H$, and we denote the counit map in $H$ by $\pi:H\to \ka$. We also fix the number $e:=\dim_{\ka}(H)$, which is called the \emph{order} of the finite group scheme $\fg$.
\begin{definition}
We say that $R$ is a \emph{$\fg$-ring}, if $R$ is a $\ka$-algebra together with a group scheme action of $\fg$ on $\spec(R)$.

Similarly, we define \emph{$\fg$-fields}, \emph{$\fg$-morphisms}, \emph{$\fg$-extensions}, etc.
\end{definition}

\begin{remark}\label{remiter}
Let $R$ be a $\ka$-algebra.
\begin{enumerate}
\item A $\fg$-ring structure on $R$ is the same as a $\ka$-algebra map $\partial:R\to R\otimes_{\ka} H$ such that $\pi\circ \partial=\id_R$  (the counit condition) and the following diagram commutes (the coassociativity condition):
$$
\xymatrixcolsep{4pc}
\xymatrix{R\otimes_{\ka} H \ar[r]^-{\id\otimes\mu}& (R\otimes_{\ka} H)\otimes_{\ka} H \\
R \ar[u]^-{\partial} \ar[r]^-{\partial} & R\otimes_{\ka} H \ar[u]_-{\partial\otimes\id}.}
$$
If $R$ is a $\fg$-ring and we want to emphasize the group scheme action on $R$, then we say: ``$(R,\partial)$ is a $\fg$-ring''.

\item A $\ka$-algebra map $\partial:R\to R\otimes_{\ka} H$ satisfying only the counit condition (from Item $(1)$ above) was called an \emph{$H$-operator} in \cite{BHKK}.
This notion originated from \cite{MS, MS2} and should be thought of as a ``free operator'' in the sense that it is free from the coassociativity condition (or the \emph{iterativity condition}) given by the multiplication in $\fg$.
\end{enumerate}
\end{remark}

\begin{example}\label{ex1}
We give below several examples of finite group schemes and we interpret their actions.
\begin{enumerate}
\item For a finite group $G$, we denote by $G_{\ka}$ the finite group scheme over $\ka$ corresponding to $G$, that is:
$$G_{\ka}:=\spec(\mathrm{Func}(G,\ka)),$$
where the comultiplication in $\mathrm{Func}(G,\ka)$ comes from the group operation in $G$. Then, $G_{\ka}$-rings coincide with $\ka$-algebras with actions of the group $G$ by $\ka$-algebra automorphisms.

\item Let us assume that $\mathrm{char}(\ka)=p>0$. A \emph{truncated group scheme} (see \cite{Chase2}) over $\ka$ is a finite group scheme whose underlying scheme is isomorphic to
    $$\spec\left(\ka[X_1,...,X_f]/\left(X_1^{p^m},\ldots,X_f^{p^m}\right)\right)$$
for some $f,m>0$. If the base field $\ka$ is perfect, then each infinitesimal group scheme over $\ka$ is truncated (see e.g. \cite[Corollary 6.3, p. 347]{DeGa}). If $\fg$ is a truncated group scheme, then $(R,\partial)$ is a $\fg$-ring if and only if $\partial$ is a $\fg$-derivation on $R$ in the sense of \cite[Definition 3.8]{HK1} (see also \cite[Definition 3.9]{HK1} for an interpretation in terms of $m$-truncated $f$-dimensional Hasse-Schmidt derivations on $R$ over $\ka$ satisfying a ``$\fg$-iterativity'' rule).

The simplest example of a truncated group scheme is the kernel of the Frobenius morphisms on the additive group scheme over $\ka$ and we denote this kernel by $\ga[1]$. Then, a $\ga[1]$-ring is the same as a $\ka$-algebra $R$ together with a $\ka$-derivation $d$ such that $d^{(p)}=0$ (the composition of $d$ with itself $p$ times).

\item Let us assume that $\fg=\fg_1\times \fg_2$. Then, an action of $\fg$ on a $\ka$-scheme $X$ may be understood as an action of $\fg_1$ on $X$ together an action of $\fg_2$ on $X$ such that these two actions commute with each other.

    In the case when $\mathrm{char}(\ka)=p>0$ and $\fg=\ga[1]\times G_{\ka}$ for a finite group $G$, a $\fg$-ring $R$ is the same as a derivation $d$ on $R$ such that $d^{(p)}=0$ (see Item $(2)$ above) together with an action of $G$ on $R$ by $\ka$-algebra differential automorphisms.

\item Assume that $\fg$ is a finite \'{e}tale group scheme. Then, there is a finite Galois extension $\ka\subseteq \ka'$ such that the group scheme $$\fg_{\ka'}:=\fg\times _{\spec(\ka)}\spec(\ka')$$
    over $\ka'$ is constant, that is: there is  finite group $G$ such that $\fg_{\ka'}\cong G_{\ka'}$ (see e.g. \cite[Section 6]{Water}).

     As an explicit non-constant example, let us take $\fg$ as the finite group scheme of third roots of unity over $\Qq$. That is:
$$\fg=\spec(\Qq[\varepsilon]),\ \ \ \ \Qq[\varepsilon]:=\Qq[X]/(X^3-1),$$
and the counit and the comultiplication in $\Qq[\varepsilon]$ are given by:
$$\pi(\varepsilon)=1,\ \ \ \ \ \mu(\varepsilon)=\varepsilon\otimes \varepsilon.$$
We also have that:
$$\fg_{\Qq(\zeta)}\cong \left(\Zz/3\Zz\right)_{\Qq(\zeta)},$$
where $\zeta=\exp(2\pi i/3)$ is the third primitive root of unity in $\Cc$.
\\
The following $\Qq$-algebra map:
$$\partial:\Qq(\sqrt[3]{2})\ra \Qq(\sqrt[3]{2})\otimes_{\Qq}\Qq[\varepsilon],\ \ \ \ \ \partial(\sqrt[3]{2}):=\sqrt[3]{2}\otimes \varepsilon$$
gives $\Qq(\sqrt[3]{2})$ a $\fg$-field structure. This action has an obvious functorial description on $R$-rational points ($R$ is a $\Qq$-algebra): the group of third roots of unity in $R$ naturally acts by the ring multiplication on the set of third roots of $2$ in $R$.
\end{enumerate}
\end{example}
As mentioned in the Introduction, we will provide now a comparison between the notion of a $\fg$-ring considered in this paper and the notion of a $\underline{\mathcal{D}}$-ring for a generalized Hasse-Schmidt system $\underline{\mathcal{D}}$ from \cite{MS1}. In short, the notion considered in \cite{MS1} is much more general than ours. A generalized Hasse-Schmidt system $\underline{\mathcal{D}}$ is an inverse system of schemes with scheme morphisms between them, which may be though of as ``partial group scheme structures''. Such an idea is also considered in \cite{Kam}, where $\underline{\mathcal{D}}$ is replaced with a basically equivalent notion of a \emph{formal group scheme}, which is a direct system of finite schemes together with some morphisms between them such that they become a group operation on the level of the formal limit. In our case, we consider just a constant system consisting of one finite group scheme $\fg$ (and the identity maps).

\subsection{First-order set-up for finite group scheme actions}\label{secfo}
In this subsection, we will describe a first-order language $L_{\fg}$ and an $L_{\fg}$-theory $\fg$-DF such that models of $\fg$-DF are exactly $\fg$-fields. If we skip the ``iterativity condition'' (see Remark \ref{remiter}(2)), then we are exactly in the set-up from \cite{BHKK}. Let us recall this set-up briefly.

We fix $\{b_0,\ldots,b_{e-1}\}$, which is a $\ka$-basis of $H$ such that $\pi(b_0)=1$ and $\pi(b_i)=0$ for $i>0$. Then, for a $\ka$-algebra $R$, any map $\partial:R\to R\otimes_kH$ can be identified with a sequence
$$\partial_0,\partial_1,\ldots,\partial_{e-1}:R\to R.$$
The map $\partial$ satisfies the counit condition (see Remark \ref{remiter}(1)) if and only if $\partial_0=\id_R$. To express the coassociativity condition, we need to use the structure constants for the comultiplication on $H$, which for $i,j,l<e$ are the elements $c^{i,j}_l\in \ka$ such that the following equality holds:
$$\mu(b_l)=\sum\limits_{i,j<e}c^{i,j}_l b_i\otimes b_j.$$
Then, $(R,\partial)$ satisfies the coassociativity condition if and only if for any $r\in R$, we have:
$$\partial_i\partial_j(r)=\sum_{l=0}^{e-1}c^{i,j}_l\partial_l(r).$$
Let $L_{\fg}$ be the language of rings expanded by $e-1$ unary function symbols and by extra constant symbols for the elements of $\ka$. It is clear by what is written above, that there is an $L_{\fg}$-theory, which we denote by $\fg$-DF, whose models are exactly $\fg$-fields. We are aiming towards showing that this theory has a model companion (Theorem \ref{gdcfexists}) and that this model companion is simple (Corollary \ref{simple}).
\begin{example}\label{ex2}
We discuss here model-theoretic properties of $\fg$-fields from Example \ref{ex1}.
\begin{enumerate}
\item If $\fg=G_{\ka}$ for a finite group $G$, then the theory $\fg$-DF coincides with the theory of\emph{ $G$-transformal fields} from \cite{HK3}. It is proved in \cite{HK3} that the theory of $G$-transformal fields has a model companion (denoted $G$-TCF), which is supersimple of finite rank coinciding with the order of $G$.

\item If $\mathrm{char}(\ka)=p>0$ and $\fg$ is a truncated group scheme, then the theory $\fg$-DF was already introduced (with the same name in this truncated context) in \cite{HK}. It is shown in \cite{HK} that the theory $\fg$-DF has a model companion (denoted $\fg$-DCF), which is strictly stable.

\item If $\fg=\fg_1\times \fg_2$, where $\fg_1$ is non-trivial truncated and $\fg_2$ is non-trivial, then, as for as we know, the model theory of $\fg$-fields has not been considered before. If the finite group scheme  $\fg_2$ is constant, then this theory fits to the set-up considered by the first author in \cite{Hoff3}, which will be explained in Section \ref{secapp}.

\item If $\fg$ is a finite \'{e}tale group scheme which is not constant, then, as far as we know, the iterativity rules (and the ``Leibniz rules'') coming from such actions have not been studied yet. We consider the action from Example \ref{ex1}(4) and we use the notation from there. We have to fix a ``good basis'' $b_0,b_1,b_2$ first such that $\pi(b_0)=1$ and $\pi(b_1)=\pi(b_2)=0$. This choice is quite arbitrary, we have chosen the following:
    $$b_0:=\frac{1+\varepsilon+\varepsilon^2}{3},\ \ \ \  b_1:=\frac{\varepsilon-1}{3},\ \ \ \ b_2:=\frac{\varepsilon^2-1}{3}.$$
\begin{enumerate}
\item We have the following multiplication table:
     \begin{IEEEeqnarray*}{rClCrCl}
b_0^2 &=& b_0, &\qquad& b_1^2 &=& \frac{-2b_1+b_2}{3},\\
b_0b_1 &=& 0, &\qquad& b_2^2 &=& \frac{b_1-2b_2}{3},\\
b_0b_2 &=& 0, &\qquad& b_1b_2 &=& \frac{-b_1-b_2}{3}.
     \end{IEEEeqnarray*}
 Then, the multiplicative rule for $\fg$-operators (with respect to this basis) is the following one:
 $$\partial_1(xy)=\frac{-2\partial_1(x)\partial_1(y)-\partial_1(x)\partial_2(y)-\partial_2(x)\partial_1(y)+\partial_2(x)\partial_2(y)}{3},$$
 $$\partial_2(xy)=\frac{\partial_1(x)\partial_1(y)-\partial_1(x)\partial_2(y)-\partial_2(x)\partial_1(y)-2\partial_2(x)\partial_2(y)}{3}.$$

 \item  We have the following comultiplication table:
$$\mu(b_0)=b_0\otimes b_0 + 2b_1\otimes b_1 - b_1\otimes b_2 - b_2\otimes b_1 + 2b_2\otimes b_2,$$
$$\mu(b_1)=b_0\otimes b_1 + b_1\otimes b_0 + b_1\otimes b_1 - b_1\otimes b_2 - b_2\otimes b_1,$$
$$\mu(b_2)=b_0\otimes b_2 + b_2\otimes b_0 + b_2\otimes b_2 - b_1\otimes b_2 - b_2\otimes b_1,$$
which gives the following $\fg$-iterativity rules ($\partial_0=\id$):
\begin{IEEEeqnarray*}{rCl}
\partial_1\circ \partial_1 &=& 2\partial_0+\partial_1,\\
\partial_2\circ \partial_2 &=& 2\partial_0+\partial_2,\\
\partial_1\circ \partial_2 &=& -\partial_0-\partial_1-\partial_2=\partial_2\circ \partial_1.
\end{IEEEeqnarray*}

\item  It is interesting to see how the rules from Item (b) above transform into the rules for the usual action of $\Zz/3\Zz$ in the case when our $\fg$-field contains the primitive third root of unity $\zeta$ in the constants of the action. In such a case, we set the following new basis:
 $$b_0':=b_0=\frac{1+\varepsilon+\varepsilon^2}{3},\ \ \ \ b_1'=\frac{1+\zeta^2\varepsilon+\zeta\varepsilon^2}{3},\ \ \ \ b_2'=\frac{1+\zeta\varepsilon+\zeta^2\varepsilon^2}{3}.$$
 It is easy to see that it is an ``orthogonal basis'', that is $b_i'b_j'=\delta^i_j$ (the Kronecker delta) and that the comultiplication on $b_0',b_1',b_2'$ gives the constant group scheme coming from the group $\Zz/3\Zz$. Then, we have the following transformation rules explaining how a $\fg$-action given by $\partial_1,\partial_2$ becomes an action of $\Zz/3\Zz$ given by $\partial_1',\partial_2'$:
 $$\partial_1'=\frac{\zeta^2\partial_1-\zeta\partial_2}{\zeta-\zeta^2},\ \ \ \ \ \ \ \ \partial_2'=\frac{-\zeta\partial_1+\zeta^2\partial_2}{\zeta-\zeta^2}.$$
\end{enumerate}
\end{enumerate}
\end{example}

\subsection{Prolongations with respect to finite group scheme actions}
We fix in this subsection a $\fg$-field $(K,\partial)$. We will recall the notion of a prolongation from \cite{MS}, which is closely related to the Weil restriction of scalars. We need the following definition first.
\begin{definition}
For a $K$-algebra $R$, we consider a ``$\partial$-twisted'' $K$-algebra structure on $R\otimes_{\ka} H$, which is the $K$-algebra structure given by the following composition:
$$\xymatrix{K \ar[r]^-{\partial} &   K\otimes_{\ka} H \ar[r]^{} &   R\otimes_{\ka} H.}$$
We denote the ring $ R\otimes_{\ka} H$ with the above $\partial$-twisted $K$-algebra structure by $R\otimes_{\ka}^{\partial} H$.
\end{definition}
\begin{remark}\label{remfunc}
The functor $R\mapsto R\otimes_{\ka}^{\partial} H$ controls the $\fg$-ring extensions of $K$. More precisely, a map
$$\partial':R\ra R\otimes_{\ka}^{\partial} H$$
is a $K$-algebra homomorphism if and only if $(R,\partial')$ is a $\fg$-ring extension of $(K,\partial)$.
\end{remark}
By a \emph{$K$-variety}, we always mean a reduced $K$-irreducible algebraic subvariety of $\Aa^n_K$ for some $n>0$.
\begin{definition}\label{defprol}
Let $V$ be a $K$-variety.
\begin{enumerate}
\item A $K$-variety $\nabla(V)$ is called a \emph{$\partial$-prolongation} of $V$, if for each $K$-algebra $R$, there is a functorial bijection:
$$\nabla(V)(R) \longleftrightarrow V\left(R\otimes_{\ka}^{\partial} H\right).$$

\item For any $K$-algebra $R$, the counit map $\mu:H\to \ka$ induces a natural $K$-algebra map $R\otimes_{\ka}^{\partial} H\to R$, which yields the following functorial morphism:
    $$\pi_V:\nabla(V)\ra V.$$
\end{enumerate}
\end{definition}
From now on, we will regard the natural bijection from Definition \ref{defprol}(1) as the identity map.
\begin{remark}\label{itobs}
We collect here several observations regarding the notion of a $\partial$-prolongation.
\begin{enumerate}
\item As explained in \cite{MS}, $\partial$-prolongations exist and they coincide with the Weil restriction to $K$ of the base extension from $K$ to
$K\otimes_{\ka}^{\partial} H$.

\item On the level of rings, the $\partial$-prolongation functor is a left-adjoint functor to the following functor:
$$\mathrm{Alg}_K\ni R\mapsto R\otimes_{\ka}^{\partial} H\in \mathrm{Alg}_K.$$

\item It is good to point out that the notion of a $\partial$-prolongation does not depend on the group scheme structure on $\fg$. In particular, $\partial$-prolongations also exist if $\partial$ is a ``non-iterative'' $H$-operator on $K$ in the sense of Remark \ref{remiter}(2). However, we notice below that if $(K,\partial)$ is a $\fg$-field (which is our assumption here), then the $\partial$-prolongation functor still has some in-built iterativity properties (see Definition \ref{defc} and Remark \ref{last}).
\end{enumerate}
\end{remark}
\begin{example}
For finite group schemes considered in Example \ref{ex1}(1) and (2), the $\partial$-prolongation functor has a natural description. Let $V$ be a $K$-variety.
\begin{enumerate}
\item If $\fg=G_{\ka}$ for a finite group $G$, then $G$ acts on $K$ by field automorphisms. Let us recall that for any $\sigma\in \mathrm{Aut}(K)$, we have the ``$\sigma$-twisted'' variety $V^{\sigma}$ defined as:
$$V^{\sigma}:=V\times_{\spec(K)}(\spec(K),\spec(\sigma)).$$
If $G=\{g_1,\ldots,g_e\}$, then we have:
$$\nabla(V)=V^{g_1}\times \ldots \times V^{g_e}$$
(see \cite[Remark 2.7(4)]{HK3}).

\item If $\mathrm{char}(\ka)=p>0$ and $\fg$ is a finite truncated group scheme, then $\nabla(V)$ is a torsor of a \emph{higher tangent bundle} of $V$. For a more explicit description, let us consider the simplest case when $\fg=\ga[1]$ and $p=2$. Then, we know that $\partial$ corresponds to a derivation $d$ on $K$ such that $d\circ d=0$. If $V$ is defined over the kernel of $d$, then $\nabla(V)$ coincides with the tangent bundle of $V$.
\end{enumerate}
\end{example}
\begin{definition}\label{defpv}
Let $V$ be a $K$-variety and $(R,\partial')$ be a $\fg$-ring extension of $(K,\partial)$. By Remark \ref{remfunc}, the map
$$\partial':R\ra R\otimes_{\ka}^{\partial} H$$
is a $K$-algebra homomorphism, hence it induces the following map, which is natural in $V$ (but it is not a morphism!):
$$\partial'_V:=V(\partial'):V(R)\ra V\left(R\otimes_{\ka}^{\partial} H\right)=\nabla(V)(R).$$
\end{definition}
It is easy to observe the following.
\begin{remark}\label{easy}
Let us assume that we are in the situation from Definition \ref{defpv}.
\begin{enumerate}
\item The following diagram is commutative:
$$
\xymatrixcolsep{4pc}
\xymatrix{V(R) \ar[r]^-{\partial'_V}& \nabla(V)(R) \\
V(K) \ar[u]^-{\subseteq} \ar[r]^-{\partial_V} & \nabla(V)(K) \ar[u]_-{\subseteq}.}
$$

\item Let as assume that $V\subseteq \Aa^n$ and let $\partial_1,\ldots,\partial_{e-1}$ (see Section \ref{secfo}) act coordinate-wise on $\Aa^n(R)=R^n$.
Then, the map $\partial_V':V(R)\to \nabla(V)(R)$ is given by:
$$\partial_V'(r)=\left(r,\partial_1'(r),\ldots,\partial_{e-1}'(r)\right),$$
and the map $\pi_V$ corresponds to the projection on the first ($n$-tuple) coordinate.
\end{enumerate}
\end{remark}
We will describe now a morphism $\nabla(V)\ra \nabla\left(\nabla(V)\right)$, which reflects the ``iterativity'' of $\partial$ on the level of the $\partial$-prolongation (it was mentioned in Remark \ref{itobs}(3)). We need a lemma first.
\begin{lemma}\label{doubletwist}
The map
$$\id\otimes\mu: R\otimes_{\ka}^{\partial} H\ra \left(R\otimes_{\ka}^{\partial} H\right)\otimes_{\ka}^{\partial} H$$
is a  morphism of $K$-algebras.
\end{lemma}
\begin{proof}
Let $\iota:K\to R$ be the $K$-algebra structure map. By Remark \ref{remiter}(1), the following diagram commutes:
$$
\xymatrixcolsep{4pc}
\xymatrix{R\otimes_{\ka} H \ar[r]^-{\id\otimes\mu} & R\otimes_{\ka} H\otimes_{\ka} H \\
K\otimes_{\ka} H \ar[r]^-{\id\otimes\mu} \ar[u]^-{\iota\otimes\id}& K\otimes_{\ka} H\otimes_{\ka} H \ar[u]_-{\iota\otimes\id\otimes \mu}\\
K \ar[r]^-{\partial} \ar[u]^-{\partial} & K\otimes_{\ka} H \ar[u]_-{\partial\otimes\id}.}
$$
Since the ring homomorphism $(\iota\otimes\id)\circ\partial$ gives the $K$-algebra structure of $R\otimes_{\ka}^{\partial}H$ and the ring homomorphism $(\iota\otimes\id\otimes\id)\circ(\partial\otimes\id)\circ\partial$ gives the $K$-algebra structure on $(R\otimes_{\ka}^{\partial}H)\otimes_{\ka}^{\partial} H$, the result follows.
\end{proof}
\begin{definition}\label{defc}
Let $V$ be an affine $K$-scheme. We define the following scheme morphism (natural in $V$)
$$c_V:\nabla(V)\ra \nabla\left(\nabla(V)\right)$$
using for any $K$-algebra $R$ the following map (well-defined by Lemma \ref{doubletwist}):
$$c_V:=V\left(\id\otimes\mu\right):V\left(R\otimes_{\ka}^{\partial} H\right)\ra V\left(\left(R\otimes_{\ka}^{\partial} H\right)\otimes_{\ka}^{\partial} H\right).$$
\end{definition}
\begin{remark}\label{cfor}
\begin{enumerate}
\item In the set-up of Remark \ref{easy}(2), the morphism $c_V$ above is given on rational points by the following formula:
$$c_V(x_0,\ldots,x_{e-1})=\left(\sum_{l=0}^{e-1}c^{i,j}_l x_l\right)_{i,j<e}.$$

\item In terms of the ``constant'' generalized Hasse-Schmidt system $\underline{\mathcal{D}}$ from \cite{MS1} (see the last paragraph of Section \ref{secdefex}), the morphism $c_V$ above coincides with the morphism $\hat{\Delta}$ from \cite[Remark 2.18]{MS}.
\end{enumerate}
\end{remark}
\begin{lemma}\label{l2}
We have the following:
$$\partial_{\nabla(V)}\circ \partial_V = c_V \circ \partial_V.$$
\end{lemma}
\begin{proof}
By Definitions \ref{defpv} and \ref{defc}, we have:
$$\partial_V=V(\partial),\ \ \ \ \ c_V:=V\left(\id\otimes\mu\right).$$
Therefore, the statement we are showing follows from applying $V$ (regarded as the functor of rational points) to the commutative diagram from Remark \ref{remiter}(1).
\end{proof}
By the definition of $\nabla(V)$, the morphisms $V\to \nabla(V)$ correspond to the $K$-algebra maps $K[V]\to K[V]\otimes_{\ka}^{\partial} H$. The following result explains when such morphisms correspond to $\fg$-ring structures on $K[V]$ extending $\partial$ on $K$ (equivalently: $\fg$-actions on $V$).
\begin{lemma}\label{iterdual}
Let us fix a morphism $\varphi:V\to \nabla(V)$. Then, the following are equivalent.
\begin{enumerate}
\item  The corresponding map $K[V]\to K[V]\otimes_{\ka}^{\partial} H$ makes $K[V]$ a $\fg$-ring extension of $K$.

\item We have $\pi_V\circ \varphi=\id_V$ and the following diagram is commutative:
$$
\xymatrixcolsep{4pc}
\xymatrix{\nabla(V) \ar[r]^-{c_V}& \nabla(\nabla(V)) \\
V \ar[u]^-{\varphi} \ar[r]^-{\varphi} & \nabla(V) \ar[u]_-{\nabla(\varphi)}.}
$$
\end{enumerate}
\end{lemma}
\begin{proof}
It follows directly from the adjointness property defining the functor $\nabla$ and from Remark \ref{remiter}(1).
\end{proof}
\begin{example}\label{standard}
By Lemma \ref{iterdual}, we obtain that the morphism
$$c_V:\nabla(V)\ra \nabla(\nabla(V))$$
gives the ring $K[\nabla(V)]$ a canonical $\fg$-ring structure $\widetilde{\partial}=(\widetilde{\partial}_0,\ldots,\widetilde{\partial}_{e-1})$. This canonical structure is dual (or rather ``adjoint'') to the one considered in \cite[Proposition 3.4]{HK1}.

Let us give an explicit description of this structure for $V=\Aa^1$, so
$$K[\nabla(V)]=K[X_0,\ldots,X_{e-1}]=:K[X_{<e}].$$
For any $K$-algebra $R$, we have (see Lemma \ref{doubletwist}):
$$\nabla(\Aa^1)(R)=R\otimes^{\partial}_{\ka}H,\ \ \ \ \ \nabla(\nabla(\Aa^1))(R)=\left(R\otimes^{\partial}_{\ka}H\right)\otimes^{\partial}_{\ka}H,$$
$$c_{\Aa^1}=\id\otimes \mu:R\otimes^{\partial}_{\ka}H\ra \left(R\otimes^{\partial}_{\ka}H\right)\otimes^{\partial}_{\ka}H.$$
Since for any $n>0$, we have that $\nabla(\Aa^n)=\Aa^{ne}$, by Remark \ref{cfor} we get the following interpretation (using the fixed basis $b_0,\ldots,b_{e-1}$ of $H$ over $\ka$):
$$c_{\Aa^1}:\nabla(\Aa^1)(R)=R^e\ra \nabla(\nabla(\Aa^1))(R)=R^{e^2},\ \ \ \ \ c_{\Aa^1}(r_0,\ldots,r_{e-1})=\left(\sum_{l=0}^{e-1}c^{i,j}_lr_l\right)_{i,j<e}.$$
We get the corresponding map on the coordinate rings:
$$c_{\Aa^1}^*:K\left[X_i^{(j)}\ |\ i,j<e\right]\ra K[X_{<e}],\ \ \ \ c_{\Aa^1}^*\left(X_i^{(j)}\right)=\sum_{l=0}^{e-1}c^{i,j}_lX_l.$$
The adjoint map to this last map above gives the canonical $\fg$-ring structure on $K[\nabla(\Aa^1)]$:
$$\widetilde{\partial}=\left(c_{\Aa^1}^*\right)^{\dagger}:K[X_{<e}]\ra K[X_{<e}]\otimes^{\partial}_{\ka}H,\ \ \ \ \widetilde{\partial}(X_j)=\sum_{j=0}^{e-1}\left(\sum_{l=0}^{e-1}c^{i,j}_lX_l\right)\otimes b_j$$
In terms of the operators $\widetilde{\partial}_0,\ldots,\widetilde{\partial}_{e-1}$, we get:
$$\widetilde{\partial}_i(X_j)=\sum_{l=0}^{e-1}c^{i,j}_lX_l.$$
We consider below several explicit examples.
\begin{enumerate}
\item Let us assume that $\mathrm{char}(\ka)=2$ and
$$\fg=\ga[1]=\spec(\ka[v]),\ \ \ \ \mu(v)=v\otimes 1+ 1\otimes v,\ \ \ \ \ b_0=1,b_1=v;$$
where $\ka[v]:=\ka[X]/(X^2)$. In this case, we have:
\begin{IEEEeqnarray*}{rCl}
\widetilde{\partial}(X_0) &=& X_0\otimes b_0 + X_1\otimes b_1,\\
\widetilde{\partial}(X_1) &=& X_1\otimes b_0,
\end{IEEEeqnarray*}
and we see that:
$$\widetilde{\partial}_1(X_0)=X_1,\ \ \ \ \widetilde{\partial}_1(X_1)=0.$$

\item We still assume that $\ch(\ka)=2$, the ring $\ka[v]$ is as above, and we take:
$$\fg=\gm[1]=\spec(\ka[v]),\ \ \ \ \mu(v)=v\otimes 1+ 1\otimes v+v\otimes v,\ \ \ \ \ b_0=1,b_1=v.$$
We obtain the following:
$$\widetilde{\partial}_1(X_0)=X_1,\ \ \ \ \widetilde{\partial}_1(X_1)=X_1.$$

\item If $G=\{g_0,\ldots,g_{e-1}\}$ is a finite group and $\fg=G_{\ka}$, then we obtain a canonical $G$-action on $K[X_{<e}]$ such that for all $i,j,k<e$, we have:
$$g_i\cdot X_j=X_k$$
if and only if $g_ig_j=g_k$.
\end{enumerate}
\end{example}
\begin{remark}\label{last}
We would like to say few words about an explicit construction of the $\partial$-prolongation functor understood as the left-adjoint functor from Remark \ref{itobs}(2). If $R=K[\overline{X}]/I$ is an arbitrary $K$-algebra, then there is a ``non-iterative $H$-operator $\overline{\partial}$ from $K[\overline{X}]$ to $K\{\overline{X}\}:=K[\overline{X},\overline{X}_1,\ldots,\overline{X}_{e-1}]$'' defined by $\overline{\partial}_i(\overline{X}):=\overline{X}_i$, and the $\partial$-prolongation of $R$ coincides with $K\{\overline{X}\}/(\overline{\partial}(I))$.

In our ``iterative'' situation, the operator $\overline{\partial}$ extends to the $\fg$-ring structure $\widetilde{\partial}$ as explained in Example \ref{standard}. Then, we have that $\overline{\partial}(I)=\widetilde{\partial}(I)$ is an ``$\fg$-ideal'', hence the quotient $K\{\overline{X}\}/(\overline{\partial}(I))$ gets a natural $\fg$-ring structure, which coincides with the one described in  Example \ref{standard}.
\end{remark}

\section{Model companion}\label{secexist}

In this section, we prove the existence of a model companion of the theory of $\fg$-fields for an arbitrary finite group scheme $\fg$ over the field $\ka$.
The proofs in this section follow the lines of the proofs from \cite{K3} and \cite{HK}, however, we made some improvements here in order to make these proofs simpler and ``coordinate-free''.

For the next result, we need to introduce the notion of constants of a $\fg$-ring (these are the same constants as in \cite[Definition 2.2(3)]{BHKK}).
\begin{definition}\label{defcon}
Let $(R,\partial)$ be a $\mathfrak{g}$-ring. We denote by $R^{\fg}$ the \emph{ring of constants} of $(R,\partial)$, that is:
$$R^{\fg}:=\{x\in R\ |\ \partial(x)=x\},$$
where $\partial$ is understood as a map $R\to R\otimes_{\ka}H$ and $R$ is considered as a subring of $R\otimes_{\ka}H$.
\end{definition}
\begin{remark}\label{integral}
We note here several properties of the ring extension $R^{\fg}\subseteq R$.
\begin{enumerate}
\item By \cite[Theorem 4.2.1]{Montgomery}, the ring extension $R^{\fg}\subseteq R$ is integral (see also the proof of \cite[Section 12, Theorem 1]{MuAbel}).

\item The extension $R^{\fg}\subseteq R$ need not be finite, that is $R$ is not necessarily a finitely generated $R^{\fg}$-module (by Item $(1)$ above, it is equivalent to saying that $R$ is not necessarily a finitely generated $R^{\fg}$-algebra). Finite generation fails in the simplest non-trivial case, that is for the constant group $\fg=(\Zz/2\Zz)_{\ka}$ as demonstrated in \cite[Example 5.5]{montgomeryfixed} (a Noetherian domain $R$  with an action of  $G=\Zz/2\Zz$ such that $R^G$ is not Noetherian and $R$ is not a finite $R^G$-module).

    Finite generation holds in the case when $R$ is a finitely generated $\ka$-algebra (it is quite easy to show knowing that the ring extension $R^{\fg}\subseteq R$ is integral, see e.g. \cite[Lemma 10, page 49]{serre1988algebraic}), but this is not a good assumption for us.

\item We will show in Section \ref{secprop} (Theorem \ref{fgfinite}) that if $K$ is a field, then the field extension $K^{\fg}\subseteq K$ \emph{is} finite and of degree bounded by $e$, as in the case of group actions. It is hard to believe that this result is new, but we were unable to find it in the literature. The closest result we could find is \cite[Theorem 8.3.7]{Montgomery}, but there is an extra assumption there saying that the extension $K^{\fg}\subseteq K$ should be \emph{Hopf Galois} (see \cite[Definition 8.1.1]{Montgomery}) and, as \cite[Example 8.1.3]{Montgomery} shows, the field extension $K^{\fg}\subseteq K$ need not be Hopf Galois in general.
\end{enumerate}
\end{remark}
\begin{lemma}\label{tenfra}
Suppose that $R$ is a $\mathfrak{g}$-ring and $R\to S,R\to T$ are homomorphisms of $\mathfrak{g}$-rings. Then, we have the following.
\begin{enumerate}
\item There is a unique $\mathfrak{g}$-ring structure on $S\otimes_RT$ such that the natural maps $S\to S\otimes_RT,T\to S\otimes_RT$ are homomorphisms of $\mathfrak{g}$-rings.

\item If $R$ is a domain, then the $\mathfrak{g}$-ring structure on $R$ extends uniquely to its field of fractions, which we denote by $\mathrm{Frac}(R)$.
\end{enumerate}
\end{lemma}
\begin{proof}
Item $(1)$ is an instance of the following very general fact: if a group $G$ in a category acts on the objects $X$ and $Y$, then there is a unique $G$-action on $X\times Y$ such that the projection maps $X\times Y\to X,X\times Y\to Y$ commute with the action of $G$.

For Item $(2)$, we notice first that by Remark \ref{integral}(1) the extension of domains $R^{\fg}\subseteq R$ is integral, which implies that:
$$\mathrm{Frac}\left(R^{\fg}\right)[R]=\mathrm{Frac}(R).$$
Since tensor products commute with localizations, we get that:
$$\mathrm{Frac}\left(R^{\fg}\right)\otimes_{R^{\fg}}R\cong_R\mathrm{Frac}(R).$$
Therefore, using Item $(1)$, we obtain that there is an $\fg$-action on $\mathrm{Frac}(R)$ extending uniquely the trivial action on $\mathrm{Frac}\left(R^{\fg}\right)$ and the given action on $R$.
 \end{proof}
 \begin{remark}
 We do not know whether Lemma \ref{tenfra} holds without the finiteness assumption on the group scheme.
 \end{remark}
For the rest of this section we fix a $\fg$-field $K$ and a $K$-variety $V$. The following lemma is an extended version of \cite[Lemma 1.1(ii)]{K3}.
\begin{prop}\label{mainprop}
Assume that $W$ is a $K$-variety such that $W\subseteq \nabla(V)$ and $c_V(W)\subseteq \nabla(W)$. Then, we have the following.
\begin{enumerate}
\item There is a $\fg$-field structure $\partial'$ on $K(W)$ such that $K\subseteq K(W)$ is a $\fg$-field extension.

\item Let $b:W\to V$ denote the composition of the inclusion $W\subseteq \nabla(V)$ and the projection morphism $\pi_V:\nabla(V)\to V$. We regard $b$ as an element of $V(K(W))$. Then, we have:
    $$\partial'_V(b)\in W(K(W)),$$
    where $\partial'$ comes from Item $(1)$ above.
\end{enumerate}
\end{prop}
\begin{proof}
For the proof of Item $(1)$, we consider the morphism:
$$\varphi:=c_V|_W:W\ra \nabla(W).$$
By Example \ref{standard}, the morphism $\varphi$ satisfies the second condition from Lemma \ref{iterdual}, hence $K\subseteq K[W]$ has a natural structure of a $\fg$-ring extension. By Lemma \ref{tenfra}, the $\fg$-ring structure on $K[W]$ extends to a $\fg$-field structure, which we denote by $\partial'$, on $K(W)$.

For Item $(2)$, we notice first that the rational point $\partial'_V(b)\in \nabla(V)(K(W))$ corresponds to the following morphism:
$$\nabla(b)\circ \varphi:W\ra \nabla(V).$$
It is enough to show that this morphism coincides with the inclusion $W\subseteq \nabla(V)$, which follows from the commutativity of the following diagram:
$$
\xymatrixcolsep{4pc}
\xymatrix{W \ar[r]^-{\subseteq} \ar[d]_-{\varphi} & \nabla(V) \ar[d]^-{c_V} \\
\nabla(W) \ar[r]^-{\subseteq} \ar[rd]_-{\nabla(b)} & \nabla(\nabla(V)) \ar[d]^-{\nabla(\pi_V)}\\
 & \nabla(V)}
$$
and the facts that $\nabla(\pi_V)=\pi_{\nabla(V)}$ and $\pi_{\nabla(V)}\circ c_V=\id_{\nabla(V)}$ (since $c_V$ satisfies the counit condition from Remark \ref{remiter}(1)).
\end{proof}
We can formulate now:
\smallskip
\\
\textbf{Geometric axioms for $\mathfrak{g}$-DCF}
\\
For each positive integer $n$, suppose that $V \subseteq \Aa^n$ and $W \subseteq \nabla(V)$ are $K$-varieties. If $c_{V}(W)\subseteq \nabla(W)$, then there is $a\in  V(K)$ such that $\partial_{V}(a)\in W(K)$.
\smallskip
\\
It is standard to see that these axioms are first-order, see e.g. the discussion in \cite[Remark 2.7]{HK3}.
\begin{theorem}\label{gdcfexists}
The $\mathfrak{g}$-field $(K,\partial)$ is existentially closed if and only if $(K,\partial)$
 satisfies the geometric axioms above.
\end{theorem}
\begin{proof}
For the left-to-right implication, let us assume that the $\mathfrak{g}$-field $K$ is existentially closed and take $V,W$ as in the assumptions of the geometric axioms for $\mathfrak{g}$-DCF. As in the statement of Proposition \ref{mainprop}, we regard the projection morphism $W\to V$ as a rational point $b\in V(K(W))$. By Proposition \ref{mainprop}, $(K,\partial)\subseteq (K(W),\partial')$ is a $\fg$-field extension and $\partial'_V(b)\in W(K(W))$. By Lemma \ref{easy} and our assumption saying that the $\fg$-field $(K,\partial)$ is existentially closed, we obtain that there is $a\in V(K)$ such that $\partial_V(a)\in W(K)$.

For the right-to-left implication, let us assume that $(K,\partial)$ satisfies the geometric axioms for $\mathfrak{g}$-DCF. We take a $\fg$-field extension $(K,\partial)\subseteq (L,\partial')$ and a quantifier-free $L_{\fg}$-formula $\phi(x)$ over $K$ having a realisation $v$ in $(L,\partial')$. By the usual trick of introducing extra variables (to get rid of any negations of equalities), we can assume that for
$$V:=\mathrm{locus}_K(v),\ \ \ \ \ \ W:=\mathrm{locus}_K(\partial_V'(v))$$
the formula $\phi(x)$ is implied by the formula saying that ``$x\in V$ and $\partial_V(x)\in W$'' (the latter is clearly a first-order formula by Remark \ref{easy}(2)). By Lemma \ref{l2}, we get that $c_V(W)\subseteq \nabla(W)$, thus by the geometric axioms for $\mathfrak{g}$-DCF, the formula $\phi(x)$ has a realisation in $(K,\partial)$.
\end{proof}

\begin{example}
The geometric axioms for $\fg$-DCF generalize both the axiomatizations given in \cite{HK} and in  \cite{HK3}. We explain below how the axioms for $\fg$-DCF specialize to these two cases.
\begin{enumerate}
\item In \cite{HK}, the case of a truncated (see Example \ref{ex1}(2)) finite group
scheme $\fg$ is considered. The axioms for the theory of existentially closed $\fg$-fields from \cite[Section 5.2]{HK} have basically the same form as the axioms above. However, the idea that similar axioms as in \cite[Section 5.2]{HK} may work for an arbitrary finite group scheme $\fg$ only came to us at the final moments of writing the article \cite{HK3} (see \cite[Remark 5.1]{HK3}).

\item In \cite{HK3}, we studied model theory of actions of finite groups on fields, which correspond to actions of finite constant group schemes (see Example \ref{ex2}(1)). The geometric version of the corresponding axioms is given in \cite[Remark 2.7(2)(b)]{HK3} and the comparison to axioms from Item $(1)$ above (so, also to the axioms from Theorem \ref{gdcfexists}) is discussed in \cite[Remark 2.7(4)]{HK3}.
\end{enumerate}
\end{example}

\section{Fields of constants and simplicity}\label{secprop}
The main result of this section is Theorem \ref{simple}, which says that the theory $\mathfrak{g}$-DCF (see Theorem \ref{gdcfexists}) is simple. This result follows from a Galois-theoretic analysis of the field of constants of an existentially closed $\fg$-field: we show that this field is pseudo-algebraically closed and bounded. We also show that any $\fg$-field is a finite extension of its field of constants (as a pure field).

Our first aim is to show the finite generation result mentioned above. The next lemma holds in any category in which quotients of group actions exist (as in the case of standard group actions). We recall that quotients exist in the category of finite group schemes (see e.g. \cite[Section 16.3]{Water}).
\begin{lemma}\label{brick}
Assume that $K$ is a $\fg$-field and $\mathfrak{n}$ is a normal subgroup scheme of $\fg$. Then, we have the following.
\begin{enumerate}
\item $K$ is an $\mathfrak{n}$-field and $K^{\mathfrak{n}}$ is a $\fg/\mathfrak{n}$-field.

\item The following two subfields of invariants coincide:
$$\left(K^{\mathfrak{n}}\right)^{\fg/\mathfrak{n}}=K^{\fg}.$$
\end{enumerate}
\end{lemma}
Thanks to Lemma \ref{brick}, we can reduce the proof of our first aim to the case of a reasonably simple finite group scheme $\fg$. The next two results deal with some particular cases of such group schemes. We recall that $e$ is the order of $\fg$, that is $e=\dim_{\ka}(H)$, where $\fg=\spec(H)$.
\begin{lemma}\label{infinitesimal}
Assume that $K$ is a $\fg$-field, $\mathrm{char}(\ka)=p>0$ and $\mathfrak{g}=\ker(\fr_{\fg})$. Then, we have:
$$\left[K:K^{\fg}\right]\leqslant e.$$
\end{lemma}
\begin{proof}
The proof of \cite[Corollary 3.21]{HK} works in the case considered here as well (our $e$ here plays the role of $pe$ from \cite{HK}).
\end{proof}
\begin{lemma}\label{etale}
Assume that $K$ is a $\fg$-field and $\mathfrak{g}$ is \'{e}tale. Then we have:
$$\left[K:K^{\fg}\right]\leqslant e.$$
\end{lemma}
\begin{proof}
Let $\ka\subseteq \mathbf{l}$ be a finite Galois extension such that $\fg_{\mathbf{l}}\cong G_{\mathbf{l}}$, where $G$ is a finite group of order $e$ (see Example \ref{ex1}(4)). Since $K$ is a $\fg$-field, $K\otimes_{\ka}\mathbf{l}$ becomes a $\fg_{\mathbf{l}}$-ring, which means that the group $G$ acts on $K\otimes_{\ka}\mathbf{l}$ by $\mathbf{l}$-algebra automorphisms. By \cite[Theorem 4.4]{BHKK}, it is enough to show that $K\otimes_{\ka}\mathbf{l}$ can be generated by $e$ elements as a $(K\otimes_{\ka}\mathbf{l})^G$-module.

Since the extension $\ka\subseteq \mathbf{l}$ is finite separable, the ring $K\otimes_{\ka}\mathbf{l}$ is a reduced Artinian $K$-algebra, so it is semisimple.
By \cite[Theorem 2.19]{montgomeryfixed} and the comments about the finite number of generators after its proof (since a commutative semisimple ring is necessarily self-injective), our inequality follows.
\end{proof}
We are ready now to show our main finite generation result. As already pointed out in Remark \ref{integral}(3), we doubt whether this result is new, but we could not find any reference to it.
\begin{theorem}\label{fgfinite}
Assume that $K$ is a $\fg$-field. Then we have:
$$\left[K:K^{\fg}\right]\leqslant e.$$
\end{theorem}
\begin{proof}
By the short exact sequence of finite group schemes from the Introduction (see \cite[Section 6.7]{Water}) and Lemma \ref{brick}, we can assume that $\fg$ is connected or $\fg$ is \'{e}tale. The \'{e}tale case holds by Lemma \ref{etale}. For the connected case, by Lemma \ref{brick} again we can assume that $\ker(\fr_{\fg})=\fg$ (see \cite[Remark 2.3]{HK}), and then the proof is concluded by Lemma \ref{infinitesimal}.
\end{proof}

We recall (the reader may consult e.g. \cite[Section 11]{FrJa}) that a field $M$ is \emph{pseudo-algebraically closed}, abbreviated \emph{PAC}, if each absolutely irreducible variety over $M$ has an $M$-rational point.
\begin{prop}\label{pac}
If $K$ is an existentially closed $\fg$-field, then $K^{\fg}$ is PAC.
\end{prop}
\begin{proof}
By \cite[Proposition 11.3.5]{FrJa}, we need to show that $K^{\fg}$ is existentially closed in each regular field extension $K^{\fg}\subseteq L$. We can assume that $K$ and $L$ are algebraically disjoint over $K^{\fg}$. Let us define:
$$M:=K\otimes_{K^{\fg}} L.$$
By Lemma \ref{tenfra}(1), there is a $\mathfrak{g}$-ring structure $\partial'$ on $M$ such that $L\subseteq R^{\partial'}$. Since the field extension $K^{\fg}\subseteq L$ is regular and the field extension $K^{\fg}\subseteq K$ is finite algebraic (by Theorem \ref{fgfinite}), the ring $M$ is a field.
Then we have the following:
\begin{itemize}
\item $(K,\partial)\subseteq (M,\partial')$ is a $\mathfrak{g}$-field extension;

\item the $\mathfrak{g}$-field $(K,\partial)$ is existentially closed;

\item $L\subseteq M^{\partial'}$.
\end{itemize}
The conditions above imply that for any quantifier-free formula $\varphi(x)$ over $K^{\fg}$ in the language of fields, $\varphi(K^{\fg})\neq \emptyset$ if and only if $\varphi(L)\neq \emptyset$. Therefore, $K^{\fg}$ is existentially closed in $L$, what we needed to show.
\end{proof}
\begin{cor}\label{cor:K.PAC}
If $K$ is an existentially closed field $\fg$-field, then $K$ is PAC.
\end{cor}
\begin{proof}
Since an algebraic extension of a PAC field is a PAC field (see \cite[Corollary 11.2.5]{FrJa}), it is enough to use Proposition \ref{pac} and Theorem \ref{fgfinite}.
\end{proof}
For the boundedness part, we would like to single out first the following Galois-theoretic result, which appeared in some form in \cite{Sjo}, \cite[Lemma 3.7]{HK3}, and \cite[Lemma 2.8]{BK2}. We recall that for a field $M$, $\gal(M)$ denotes the \emph{absolute Galois group} of $M$ (consider as a topological profinite group), that is:
$$\gal(M):=\gal\left(M^{\mathrm{sep}}/M\right).$$
For the necessary background about Frattini covers, we refer the reader to \cite[Chapter 22]{FrJa}.
\begin{prop}\label{newfrat}
Suppose that $M\subseteq N$ is a finite Galois field extension. Then the following conditions are equivalent.
\begin{enumerate}
\item The restriction map
$$\mathrm{res}:\gal(M)\ra \gal(N/M)$$
is a Frattini cover.

\item For any separable field extension $M\subseteq M'$, $M'$ is not linearly disjoint from $N$ over $M$.
\end{enumerate}
\end{prop}
\begin{proof}
We will use \cite[Lemma 2.8]{BK2}, which says that for any closed subgroup $\mathcal{H}\leqslant \gal(M)$ and for
$$M':=\left(M^{\alg}\right)^{\mathcal{H}},$$
the following conditions are equivalent:
\begin{enumerate}
\item[(A)] $\mathrm{res}(\mathcal{H})=\gal(N/M)$,

\item[(B)] the restriction map
$$\gal(NM'/M')\ra \gal(N/M)$$
is an isomorphism.
\end{enumerate}
It is easy to see that Condition (B) above is equivalent to saying that $M'$ is linearly disjoint from $N$ over $M$. Since the restriction map
$$\mathrm{res}:\gal(M)\ra \gal(N/M)$$
is a Frattini cover if and only if for each proper closed subgroup $\mathcal{H}$ as above, we have $\mathrm{res}(\mathcal{H})\neq \gal(N/M)$, the result follows.
\end{proof}
Our next aim is to show that if $K$ is an existentially closed $\mathfrak{g}$-field, then $K^{\fg}$ is bounded, that is the absolute Galois group $\gal(K^{\fg})$ is small. As in the case of actions of finite groups, we will achieve this aim by identifying $\gal(K^{\fg})$ with the universal Frattini cover of a certain finite group. Similarly as in the group case, the whole strength of the existential closedness assumption is not necessary here, we will only need the property expressed in the following definition.
\begin{definition}
A $\mathfrak{g}$-field $K$ is \emph{$\mathfrak{g}$-closed}, if there are no non-trivial $\mathfrak{g}$-field extensions $K\subset L$ such that the pure field extension $K\subset L$ is algebraic.
\end{definition}
The following result is crucial for our proof of the simplicity of the theory $\fg$-DCF.
\begin{theorem}\label{ncbound}
Assume that a $\fg$-field $K$ is $\fg$-closed and $C=K^{\fg}$. Let $C\subseteq K^{\mathrm{nc}}$ be the normal closure of the field extension $C\subseteq K$.
Then, the restriction map:
$$\gal(C)\ra \gal(K^{\mathrm{nc}}/C)$$
is a Frattini cover.
\end{theorem}
\begin{proof}
Let us define:
$$C^{\mathrm{ins}}:=\left(K^{\mathrm{nc}}\right)^{\gal(K^{\mathrm{nc}}/C)}.$$
By \cite[Theorem 19.18]{Isaacs}, the extension $C\subseteq C^{\mathrm{ins}}$ is purely inseparable, the extension $C^{\mathrm{ins}}\subseteq K^{\mathrm{nc}}$ is Galois, and the following diagram commutes:
\begin{equation*}
\xymatrixcolsep{4pc}
\xymatrix{\gal(C)  \ar[r]^-{\mathrm{res}}  \ar[rd]^-{\mathrm{res}} & \gal(K^{\mathrm{nc}}/C^{\mathrm{ins}}) \ar[d]^-{=}\\
 &  \gal(K^{\mathrm{nc}}/C). }
\end{equation*}
Therefore, by Proposition \ref{newfrat}, it is enough to show that for any proper Galois field extension  $C^{\mathrm{ins}}\subset C^{\mathrm{ins}'}$, we have
that $K^{\mathrm{nc}}$ is not linearly disjoint from $C^{\mathrm{ins}'}$ over $C^{\mathrm{ins}}$. Let us assume that there is a proper Galois field extension  $C^{\mathrm{ins}}\subset C^{\mathrm{ins}'}$ such that $K^{\mathrm{nc}}$ is linearly disjoint from $C^{\mathrm{ins}'}$ over $C^{\mathrm{ins}}$.
We aim to reach a contradiction. Let $C\subseteq C'$ be the maximal separable subextension of the extension $C\subset C^{\mathrm{ins}'}$.
The situation is summarized in the following commutative diagram, where all the arrows are inclusions of fields:
\begin{equation*}
\xymatrix{  & K^{\mathrm{nc}}   &   \\
K \ar[ru]^{} & C^{\mathrm{ins}} \ar[r]^{}  \ar[u]^{} &  C^{\mathrm{ins}'}  \\
  & C \ar[lu]^{} \ar[r]^{}  \ar[u]^{}  &   C'  \ar[u]^{} . }
\end{equation*}
We have the following:
$$[C^{\mathrm{ins}'}:C^{\mathrm{ins}}]=[C^{\mathrm{ins}'}:C]_{\mathrm{sep}}=[C':C],$$
$$[C^{\mathrm{ins}}:C]=[C^{\mathrm{ins}'}:C]_{\mathrm{ins}}=[C^{\mathrm{ins}'}:C'].$$
Therefore, $C\subset C'$ is a proper Galois extension and we have:
$$\gal(C'/C)\cong \gal( C^{\mathrm{ins}'}/C^{\mathrm{ins}}),\ \ \ \ K^{\mathrm{nc}}\cap C'= C.$$
Since the extension $C\subset C'$ is Galois and $K\cap C'= C$, we obtain that $K$ is linearly disjoint from $C'$ over $C$. Hence, we have:
$$KC'\cong K\otimes_CC'$$
and by Lemma \ref{tenfra}, the field extension $K\subset KC'$ has a $\fg$-field extension structure. Since the field extension $K\subset KC'$ is proper algebraic, we get a contradiction with the assumption that $K$ is $\fg$-closed.
\end{proof}
One could wonder whether the field extension $C\subseteq K$ appearing in the statement of Theorem \ref{ncbound} is a Galois extension (as it is the case when we consider the actions of groups) or perhaps it is only a normal extension or it is only a separable extension.  The example below shows that ``anything may happen'', hence we really need to consider the fields $K^{\mathrm{nc}}$ and $C^{\mathrm{ins}}$ in the proof of Theorem \ref{ncbound} above.
\begin{example}
\begin{enumerate}
\item If $\mathrm{char}(\ka)=p>0$, $\fg=\ga[1]$, and $K$ is a $\fg$-field such that the $\fg$-action is non-trivial, then the extension $K^{\fg}\subset K$ is purely inseparable and non-trivial (see \cite[Definition 3.5 and Lemma 3.16]{HK}).

\item To see that the field extension $K^{\fg}\subseteq K$ need not be normal (even in the case when $\fg$ is \'{e}tale), one can consider the $\fg$-field $\Qq(\sqrt[3]{2})$ from the end of Example \ref{ex1}(4), where $\fg$ is the group scheme of the third roots of unity over $\Qq$. Then, we have:
    $$\left(\Qq(\sqrt[3]{2})\right)^{\fg}=\Qq$$
    and the field extension $\Qq\subset \Qq(\sqrt[3]{2})$ is clearly not normal.
\end{enumerate}

\end{example}
\begin{cor}\label{bounded}
If $K$ is an existentially closed $\fg$-field, then the field $K^{\fg}$ is PAC and bounded.
\end{cor}
\begin{proof}
The PAC part is exactly Proposition \ref{pac}.

We proceed to show the boundedness part. Since an existentially closed $\fg$-field is $\fg$-closed, we can apply Theorem \ref{ncbound} and obtain that the restriction map
$$\gal(K^{\fg})\ra \gal(K^{\mathrm{nc}}/K^{\fg})$$
is a Frattini cover. Since a Frattini cover of a topologically finitely generated profinite group (in our case, it just the finite group $\gal(K^{\mathrm{nc}}/K^{\fg})$) is finitely generated again (see \cite[Corollary 22.5.3]{FrJa}), we get that the profinite group $\gal(K^{\fg})$ is topologically finitely generated. By \cite[Proposition 2.5.1(a)]{progps}, the profinite group $\gal(K^{\fg})$ is small, therefore the field $K^{\fg}$ is bounded.
\end{proof}
\begin{theorem}\label{simple}
The theory $\mathfrak{g}-\mathrm{DCF}$ is simple.
\end{theorem}
\begin{proof}
For any model $(K,\mathfrak{g})$ of $\mathfrak{g}$-DCF, by Theorem \ref{fgfinite}, the theory of $(K,\mathfrak{g})$ is bi-interpretable with the theory of the pure field $K^{\fg}$ (the description of this bi-interpretation given in \cite[Remark 2.3]{HK3} works here as well). By Corollary \ref{bounded}, the field $K^{\fg}$ is PAC and bounded. Therefore, by \cite[Fact 2.6.7]{Kimsim}, the theory of the field $K^{\fg}$ is simple.
\end{proof}
\begin{remark}\label{ssrem}
We describe here some special cases and comment on possible improvements and generalizations.
\begin{enumerate}
\item If the group scheme $\mathfrak{g}$ is infinitesimal (see Example \ref{ex1}(2)), then the theory $\mathfrak{g}$-DCF is stable and it coincides with one of the theories considered in \cite{HK}.

\item If the group scheme $\mathfrak{g}=G_{\ka}$ is constant ($G$ is a finite group), then we have
$$\mathfrak{g}-\mathrm{DCF}=G-\mathrm{TCF},$$
where $G$-TCF is the theory considered in \cite{HK3}. In particular, this theory is supersimple of finite rank $e$ coinciding with the order of $G$.

\item Suppose that the group scheme $\fg$ is neither infinitesimal nor  \'{e}tale. It means that $\fg\neq \fg^0$ and $\fg\neq \fg_0$, where $\fg^0$ and $\fg_0$ come from the exact sequence considered in the beginning of the Introduction. In this case, using \cite[Fact 2.6.7]{Kimsim} again, the theory $\mathfrak{g}$-DCF is strictly simple (that is: simple, not stable, and not supersimple), since the PAC field $K^{\fg}$ is neither separably closed nor perfect.

\item We consider in Section \ref{secapp} finite group schemes $\fg$ of some special kind, which are of the type described in Item $(3)$ above. For such finite group schemes, the theory $\fg-\mathrm{DCF}$ considered in a certain expanded language has a finer description, which is obtained using the results from  \cite{Hoff3} (see Corollary \ref{danielcor}).

\item As we pointed out in Section \ref{secprol}, the $\fg$-rings considered in this paper are a very special case of the rings with iterative operators considered in \cite{MS}. One may wonder whether our main results (Theorems \ref{gdcfexists} and \ref{simple}) could be generalized to the context considered in \cite{MS}. Unfortunately, it looks impossible at this moment, since, for example, actions of locally finite groups on rings constitute still quite a special case of the situation from \cite{MS}. But the theory of actions of such groups on fields need \emph{not} be companionable, as a recent work of Beyarslan and the second author shows \cite{BK2}. More precisely, the case of Abelian torsion groups is considered in  \cite{BK2} and a rather unexpected algebraic condition on such groups is given in \cite[Theorem 1.1]{BK2}, which is equivalent to the companionability the theory mentioned above. For the general case of locally finite groups, it is even unclear what should be conjecturally taken for this algebraic condition.
\end{enumerate}
\end{remark}

\section{Group actions on fields of finite imperfection degree}\label{secapp}
In this section, we use the results from the previous sections to show the existence and to describe some properties of a model complete theory of actions of a fixed finite group on fields of finite imperfection degree. Let us describe first the general setting of the model-theoretic dynamics from \cite{Hoff3}, which will be used here. We start from an $L$-theory $T$ having a model companion $T^{\mc}$ and we assume that $T^{\mc}$ has elimination of quantifiers. For a fixed group $G$, we expand the language $L$ to the language $L_G$ by adding unary function symbols for elements of $G$. Then, $T_G$ is the obvious $L_G$-theory of $G$-actions by automorphisms on models of $T$, and it is an interesting question whether a model companion of $T_G$ exists (if it exists, it is denoted by $T_G^{\mc}$). In the case of $G=\Zz$, the above question is about the existence of an axiomatization of the theory of \emph{generic automorphism} of models of $T$ (see \cite{ChPi}). It is well-known (see e.g. \cite{acfa1}) that if $T$ is the theory of fields, then the theory $T_{\Zz}^{\mc}$(=ACFA) exists. In \cite{Cha02}, Chatzidakis shows the existence of the theory $T_{\Zz}^{\mc}$, where $T$ is the theory of separably closed fields of a fixed imperfection degree in the language of fields expanded by symbols for $p$-basis and $\lambda$-functions.

If $G$ is a finite group and $T$ is the theory of fields, then the theory $T_G^{\mc}$ exists (see \cite{Sjo} and \cite{HK3}); this theory was denoted by $G$-TCF in \cite{HK3}. It is a good moment now to warn the reader that the theory $T_G^{\mc}$ usually does \emph{not} imply the theory $T^{\mc}$: for example the models of $G$-TCF are not algebraically (or even separably) closed for a non-trivial finite group $G$ (see \cite{HK3}). The main aim of this section is to show a variant of Chatzidakis' existence result from \cite{Cha02} for a finite group $G$ replacing the infinite cyclic group, which also gives one more example of the class of theories satisfying the conditions from \cite{Hoff3}.
To achieve this aim, we do not follow the methods from \cite{Cha02}, but we provide a new argument using Theorem \ref{gdcfexists}.

We start with describing our version of the general set-up from \cite{Hoff3}. Let us fix a positive integer $e$ (it will \emph{not} play the role of the number $e$ from the previous sections!) and a prime number $p$. For the notion of a $p$-basis, the reader may consult \cite[Section 1.5]{Cha02}. The \emph{imperfection degree} of a field $K$ of characteristic $p$ is the cardinality of its $p$-basis. We define the language $L$ as the language of fields with the extra constant symbols $b=\{b_1,\ldots,b_{e}\}$ and the extra unary function symbols $\lambda=\{\lambda_i\;|\;i\in[p]^{e}\}$, where $[p]:=\{0,\ldots,p-1\}$. We set as $T$ the $L$-theory $\mathrm{SF}_{p,e}$, which is such that its models $(K;b,\lambda)$ satisfy the following:
\begin{itemize}
  \item $K$ is a field of characteristic $p$;
  \item the set $\{b_1,\ldots,b_{e}\}$ is a $p$-basis of $K$;
  \item for each $x\in K$, we have:
  $$x=\sum\limits_{i\in[p]^{e}}\lambda_i(x)^pb^i.$$
\end{itemize}
It is well-known that the model companion $T^{\mc}$ is the theory of separably closed fields which are models of $\mathrm{SF}_{p,e}$; this theory is denoted by $\mathrm{SCF}_{p,e}$. The theory $\mathrm{SCF}_{p,e}$ is strictly stable and has quantifier elimination (Theorem 2.3 in \cite{Del1}), so we are in the general situation from \cite{Hoff3}.

We describe now how to add dynamics to the  theory $T$.  Let $G$ be a finite group of order $e'$.
As explained above, $L_G$ denotes the language $L$ expanded by unary function symbols $\{\sigma_g\;|\;g\in G\}$ and $T_G=(\mathrm{SF}_{p,e})_G$ is the theory of group actions of $G$ on models of $\mathrm{SF}_{p,e}$ by $L$-automorphisms. To show the existence of the theory $T_G^{\mc}$,
we will use the theory $\fg-\mathrm{DCF}$ (see Theorem \ref{gdcfexists}) for $\ka=\Ff_p$ and the following finite group scheme:
$$\mathfrak{g}:=\ga^{e}[1]\times G_{\Ff_p}.$$
For the rest of this section, $\fg$ denotes the finite group scheme defined above, so we diverge again from the notation of the previous sections, since $e$ is not anymore the order of $\fg$ (which is now $p^e+e'$).

By Example \ref{ex1}(3), a $\fg$-field $(K,\partial)$ is the same as a field $K$ of characteristic $p$ with derivations $D_1,\ldots,D_e$ and with an action of the group $G$ by field automorphisms on $K$ such that the following conditions hold:
\begin{itemize}
\item the derivations $D_1,\ldots,D_e$ commute with each other;

\item for each $i\leqslant e$, we have $D_i^{(p)}=0$ (the composition of $D_i$ with itself $p$ times);

\item the action of $G$ commutes with the derivations $D_1,\ldots,D_e$.
\end{itemize}
Let us define:
$$C_K:=K^{\ga^{e}[1]}=\ker\left(D_1\right)\cap \ldots \cap \ker\left(D_e\right).$$
We clearly have $K^p\subseteq C_K$ and the next result shows that there is a $\fg$-field extension of $K$ where the equality holds, which will be crucial in the proof of Theorem \ref{appthm} (our aim here).
\begin{lemma}\label{lemma:strict.extension}
For any $\fg$-field $K$, there is a $\fg$-field extension $K\subseteq M$  such that $C_M=M^p$.
\end{lemma}
\begin{proof}
If $C_K=K^p$, then there is nothing to do, so let us take $t\in C_K\setminus K^p$. Then, we have:
$$G\cdot t\subset C_K\setminus K^p.$$
We take a $p$-basis $\{t_1=t,\ldots,t_k\}$ of $K^p(G\cdot t)$ over $K^p$ and we define:
$$K_1:=K\left(t_1^{1/p},\ldots,t_k^{1/p}\right).$$
By setting for any $i\leqslant e$ and $j\leqslant k$:
$$D_i\left(t_j^{1/p}\right):=0,$$
$K_1$ becomes a $\ga^e[1]$-field extension of $K$. For any $g\in G$ and any $j\leqslant k$, we define:
$$g\cdot t_j^{1/p}:=(g\cdot t_j)^{1/p},$$
which gives a unique extension of the $G$-action from $K$ to $K_1$. This makes $K_1$ a $\fg$-field extension of $K$ and by repeating this process, we can define $M$ as the union of a tower of $\fg$-field extensions of $K$.
\end{proof}
For a $\fg$-field $K$, it will be crucial to find a $p$-basis of $K$ in the subfield of $G$-invariants $K^G$, which is allowed by the next result.
\begin{lemma}\label{pbasisconst}
Suppose that the group $G$ acts on a field $K$ of characteristic $p$ and of imperfection degree $e$. Then, there is a $p$-basis of $K$ in $K^G$.
\end{lemma}
\begin{proof}
Let us denote $K_0:=K^G$ and consider the following commutative diagram of field extensions:
\begin{equation*}
\xymatrix{  & K   &   \\
K^p \ar[ru]^{} &  &  K_0 \ar[lu]^{} \\
  & K_0^p. \ar[lu]^{}\ar[ru]^{}  &     }
\end{equation*}
By applying the Frobenius map, we have $[K:K_0]=[K^p:K_0^p]$. Therefore, we obtain:
$$[K:K^p]=[K_0:K_0^p],$$
so the imperfection degree of $K_0$ is $e$ as well. Since the field extension $K_0^p\subseteq K_0$ is purely inseparable and the extension $K_0^p\subseteq K^p$ is Galois (because the extension $K_0\subseteq K$ is Galois), we obtain that $K_0$ is linearly disjoint from $K^p$ over $K_0^p$. The linear disjointness and the fact that the imperfection degree of $K_0$ is finite and coincides with the imperfection degree of $K_0$ imply that any $p$-basis of $K_0$ is a $p$-basis of $K$ as well.
\end{proof}
\begin{prop}\label{gdcfmod}
If $K$ is a model of $\fg-\mathrm{DCF}$, then $C_K=K^p$ and the imperfection degree of $K$ is $e$.
\end{prop}
\begin{proof}
Since $K$ is an existentially closed $\fg$-field, Lemma \ref{lemma:strict.extension} implies that $C_K=K^p$. By \cite[Corollary 3.21]{HK3}, we get that $[K:C_K]\leqslant p^e$. Therefore (by the existential closedness again), it is enough to find a $\fg$-field extension $K\subseteq N$ such that the imperfection degree of $N$ is $e$.

Let us consider the field of rational functions $\Ff_p(X_1,\ldots,X_e)$ with the standard $\ga^e[1]$-field structure given by $D_i(X_j)=\delta^j_i$ (the Kronecker delta) and with the trivial action of $G$. Then, $\Ff_p(X_1,\ldots,X_e)$ becomes a $\fg$-field of imperfection degree $e$ and such that:
$$C_{\Ff_p(X_1,\ldots,X_e)}=\Ff_p\left(X_1^p,\ldots,X_e^p\right)=\left(\Ff_p(X_1,\ldots,X_e)\right)^p.$$
We consider now the following field:
$$M:=\mathrm{Frac}\left(K\otimes_{\Ff_p}\Ff_p(X_1,\ldots,X_e)\right).$$
By Lemma \ref{tenfra} (both of the items), $M$ is a $\fg$-field and it is naturally a $\fg$-field extension of $K$ and $\Ff_p(X_1,\ldots,X_e)$. Let $N$ be a $\fg$-field extension of $M$ provided by Lemma \ref{lemma:strict.extension}, so $N^p=C_N$. By \cite[Lemma 3.22]{HK3}, the field extension $\Ff_p(X_1,\ldots,X_e)\subseteq N$ is separable, hence the imperfection degree of $N$ is at least $e$.  Since $N^p=C_N$, we obtain (using \cite[Corollary 3.21]{HK3} again) that the imperfection degree of $N$ is exactly $e$.
\end{proof}
For our main bi-interpretability result, we need to consider the theory $\fg$-DCF in an expanded language. We set the following.
\begin{enumerate}
\item Let $L_{\fg,b}$ be the language $L_{\fg}$ (see Section \ref{secfo}) expanded by the constant symbols $b_1,\ldots,b_e$.

\item Let $\fg-\mathrm{DF}_b$ be the $L_{\fg,b}$-theory of $\fg$-fields such that:
\begin{enumerate}
\item $\{b_1,\ldots,b_e\}$ is a $p$-basis;

\item $\{b_1,\ldots,b_e\}$ is contained in the invariants of the action of $G$;

\item for all $i,j\leqslant e$, we have:
    $$D_i(b_j)=\delta^j_i\ \ \ \ \ \ \text{(the Kronecker delta).}$$
\end{enumerate}

\item We define the $L_{\fg,b}$-theory $\fg-\mathrm{DCF}_b$ as the theory $\fg-\mathrm{DCF}$ with the same extra conditions on the fixed $p$-basis and the $G$-invariants as in Item $(2)$ above.
\end{enumerate}
The next lemma shows that the theory $\fg-\mathrm{DF}_b$ has the expected properties.
\begin{lemma}\label{lemdfp}
We have the following.
\begin{enumerate}
\item If $K$ is a model of $\fg-\mathrm{DF}_b$, then the imperfection degree of $K$ is $e$ and $C_K=K^p$.

\item The theory $\fg-\mathrm{DF}_b$ is bi-interpretable with the theory $T_G=(\mathrm{SF}_{p,e})_G$.
\end{enumerate}
\end{lemma}
\begin{proof}
For the proof of Item $(1)$, the statement about the imperfection degree of $K$ is clear from the axiomatization of the theory $\fg-\mathrm{DF}_b$. To show that $C_K=K^p$, we consider for each $i\leqslant e$ the following subfield of $K$:
$$C_i:=\ker(D_1)\cap \ldots \cap \ker(D_i).$$
Then, we have the following tower of fields:
$$K^p\subseteq C_K=C_e\varsubsetneq C_{e-1} \varsubsetneq \ldots \varsubsetneq C_1 \varsubsetneq K,$$
where the properness of the inclusions is witnessed by $b_i\in C_{i-1}\setminus C_{i}$ (we set $C_0:=K$). Therefore, we get:
$$p^e=[K:K^p]\geqslant [K:C_K]\geqslant p^e,$$
which implies that $C_K=K^p$.

For the proof of Item $(2)$, if $(K;b,D_1,\ldots D_e,\sigma_g)_{g\in G}$ is a model of $\fg-\mathrm{DF}_b$, then $(K;b,\sigma_g)_{g\in G}$ obviously interprets a model of $T_G$, since $\lambda$-functions are definable in a field with a $p$-basis. If $(K;b,\lambda,\sigma_g)_{g\in G}$ is a model of $T_G$, then we forget about the $\lambda$-functions and define the derivations by the following formula (as in the proof of  Proposition \ref{gdcfmod}):
$$D_{j}(b_i)=\delta^j_i.$$
Then, it is standard that $(K;b,D_1,\ldots D_e,\sigma_g)_{g\in G}$ is a model of $\fg-\mathrm{DF}_b$. It is also clear that the two interpretations above are mutually inversive.
\end{proof}
The next result is crucial for the proof of Theorem \ref{appthm}.
\begin{prop}\label{finallemma}
The theory $\fg-\mathrm{DCF}_b$ is a model companion of the theory $\fg-\mathrm{DF}_b$.
\end{prop}
\begin{proof}
Clearly, any model of $\fg-\mathrm{DCF}_b$ is a model of $\fg-\mathrm{DF}_b$. We need to show that each model of $\fg-\mathrm{DF}_b$ embeds into a model of $\fg-\mathrm{DCF}_b$ and that the theory $\fg-\mathrm{DCF}_b$ is model complete.

It is easy to see the model completeness part, since the theory $\fg-\mathrm{DCF}_b$ is an expansion by constants of the model complete (by Theorem \ref{gdcfexists}) theory $\fg-\mathrm{DCF}$.

Let $(K;b,\partial)$ be a model of the theory $\fg-\mathrm{DF}_b$. By Theorem \ref{gdcfexists}, the $\fg$-field $(K,\partial)$ has a $\fg$-field extension $(K',\partial')$ which is a model of the theory $\fg-\mathrm{DCF}$. By Lemma \ref{lemdfp}, we have $C_K=K^p$. Therefore, by \cite[Lemma 3.22]{HK3}, the field extension $K\subseteq K'$ is separable. Since $b$ is a $p$-basis of $K$ and the extension $K\subseteq K'$ is separable, $b$ is $p$-independent in $K'$. By Proposition \ref{gdcfmod}, the imperfection invariant of $K'$ is $e$, so $b$ is a $p$-basis of $K'$ as well. Hence $(K';b,\partial')$ is a model of $\fg-\mathrm{DCF}_b$ and $(K';b,\partial')$ is also an extension of $(K;b,\partial)$, which finishes the proof.
\end{proof}
Before stating the main result of this section, we advice the reader to recall our notational set-up regarding the theories $T,T_G,T^{\mc},T^{\mc}_G$ from the beginning of this section.
\begin{theorem}\label{appthm}
The theory $T_G^{\mc}$ exists and it is bi-interpretable with the theory $\mathfrak{g}-\mathrm{DCF}_b$.
In particular, the theory $T_G^{\mc}$ is strictly simple.
\end{theorem}
\begin{proof}
The existence part follows from the bi-interpretability part, which is clear by Lemma \ref{lemdfp}(2) and Proposition \ref{finallemma}. The ``in particular'' part follows from Remark \ref{ssrem} and the bi-interpretability part above.
\end{proof}
\begin{remark}
As we have pointed out before, the theory $T_G^{\mc}$ need not imply the theory $T^{\mc}$ and indeed: if $G$ is non-trivial, then the underlying fields of models of $T_G^{\mc}$ are not separably closed, since the are no non-trivial actions of a finite group on a separably closed field of positive characteristic (the Artin-Schreier theorem).
\end{remark}
We can apply now the techniques from \cite{Hoff3} to obtain several model-theoretic properties of the theory $T_G^{\mc}$.
We assume that $\mathfrak{C}$ is a sufficiently saturated model of $T_G^{\mc}$ and that $\mathfrak{C}$ is embedded in a sufficiently saturated model $\mathfrak{D}$ of $T^{\mc}=\mathrm{SCF}_{p,e}$.
In other words, $\mathfrak{C}$ is an existentially closed field with a $p$-basis, $\lambda$-functions and a group action of $G$,
and $\mathfrak{D}$ is an existentially closed field with a $p$-basis and $\lambda$-functions.
We will consider the model-theoretic algebraic closure inside $\mathfrak{C}$, denoted by $\acl_G(\,\cdot\,)$,
and the model-theoretic algebraic closure inside $\mathfrak{D}$, denoted by $\acl_{\sep}(\,\cdot\,)$. By \cite[Fact 2.7]{Hr7}, if $A\subseteq \mathfrak{D}$ then  $\acl_{\sep}(A)$ coincides with the field theoretic algebraic closure in $\mathfrak{D}$ of the subfield generated by the $p$-basis and the values of the $\lambda$-functions on the elements of $A$.
\begin{cor}\label{danielcor}
Let $A,B,C$ be small subsets of $\mathfrak{C}$ such that $C\subseteq A\cap B$.
\begin{enumerate}
\item By \cite[Corollary 4.12]{Hoff3}, we have:
$$\acl_G(A)=\acl_{\sep}(G\cdot A) \cap\mathfrak{C}.$$
    \item By  \cite[Corollary 4.28]{Hoff3}, we have:
    \begin{IEEEeqnarray*}{rCl}
    A\ind_C^{T_G^{\mc}} B 
    &\qquad\iff\qquad& \acl_{\sep}(G\cdot A)\text{ and }\acl_{\sep}(G\cdot B)  \\
    & &\text{ are linearly disjoint over }\acl_{\sep}(G\cdot C),
     \end{IEEEeqnarray*}
   i.e. the forking independence in $T_G^{\mc}$ is given by the $G$-action and the forking independence in $T^{\mc}=\mathrm{SCF}_{p,e}$.
    
    \item By \cite[Remark 4.13]{Hoff3}, $T_G^{\mc}$ has ``almost quantifier elimination" (similarly as ACFA).
    
    \item  By \cite[Remark 4.36]{Hoff3}, $T_G^{\mc}$ has geometric elimination of imaginaries.
\end{enumerate}
\end{cor}
It was shown in \cite{BK} that the theory of actions of a fixed virtually free group  on fields has a model companion. An analogous result holds in the case of actions of a finite group (Theorem \ref{appthm}) or the infinite cyclic group (\cite{Cha02}) on fields of finite imperfection degree. The following natural question arises.
\begin{question}\label{q}
Does the theory of actions of a fixed virtually free group on fields of finite imperfection degree have a model companion? More generally, what is the class of groups such that the model companion above exists?
\end{question}
Let us point out that the ``more generally'' part of Question \ref{q} is wide open even in the case of group actions on arbitrary fields. 

\bibliographystyle{plain}
\bibliography{harvard}

\end{document}